\documentclass{amsart}

\usepackage{etex}
\usepackage{amsmath,amssymb,amsthm,amsfonts,mathrsfs}
\usepackage[frame,cmtip,arrow,matrix,line,graph,curve]{xy}
\usepackage{graphpap,color,paralist,pstricks}
\usepackage[mathscr]{eucal}
\usepackage{mathabx}
\usepackage[pdftex,colorlinks,backref=page,citecolor=blue]{hyperref}
\usepackage{tikz}
\usepackage{epic,eepic}
\usepackage{yfonts}
\usepackage{enumerate}
\usepackage{mathtools}

\usepackage[alphabetic]{amsrefs}

\theoremstyle{definition}
\newtheorem{theorem}{Theorem}[section] 
\newtheorem{definition}[theorem]{Definition}
\newtheorem{conjecture}[theorem]{Conjecture}
\newtheorem{lemma}[theorem]{Lemma}
\newtheorem{proposition}[theorem]{Proposition}
\newtheorem{corollary}[theorem]{Corollary}

\newtheorem*{theorem*}{Theorem}

\theoremstyle{remark}
\newtheorem{remark}[theorem]{Remark}
\newtheorem{example}[theorem]{Example}

\begin{document}

\title{Logarithmic concavity for morphisms of matroids}

\author{Christopher Eur and June Huh}

\address{University of California at Berkeley, Berkeley, CA, USA.}
\email{ceur@math.berkeley.edu}

\address{Institute for Advanced Study, Princeton, NJ, USA.}
\email{junehuh@ias.edu}

\maketitle

\begin{abstract}
Morphisms of matroids are combinatorial abstractions of linear maps and graph homomorphisms.  We introduce the notion of a basis for morphisms of matroids, and show that its generating function is strongly log-concave.  As a consequence, we obtain a generalization of Mason's conjecture on the $f$-vectors of independent subsets of matroids to arbitrary morphisms of matroids.  To establish this, we define multivariate Tutte polynomials of morphisms of matroids, and show that they are Lorentzian in the sense of \cite{BH19} for sufficiently small positive parameters.
\end{abstract}

\section{Introduction}

A \emph{matroid} $\mathrm{M}$ on a finite set $E$ is defined by its rank function $\operatorname{rk}_\mathrm{M}: 2^E \to \mathbb{N}$
satisfying the following conditions:
\begin{enumerate}[(1)]\itemsep 5pt
\item
For any $S \subseteq E$, we have $\operatorname{rk}_\mathrm{M}(S) \le |S|$.
\item For any  $S_1 \subseteq S_2  \subseteq E$,  we have
$\operatorname{rk}_\mathrm{M}(S_1) \le  \operatorname{rk}_\mathrm{M}(S_2)$.
\item For any  $S_1\subseteq E$, $S_2 \subseteq E$, we have
$\operatorname{rk}_\mathrm{M}(S_1 \cup S_2)+ \operatorname{rk}_\mathrm{M}(S_1\cap S_2) \le  \operatorname{rk}_\mathrm{M}(S_1)+ \operatorname{rk}_\mathrm{M}(S_2)$.
\end{enumerate}
A subset $S \subseteq E$ is an \emph{independent set} of $\mathrm{M}$ if $\operatorname{rk}_\mathrm{M}(S) = |S|$, and
a \emph{spanning set} of $\mathrm{M}$  if $ \operatorname{rk}_\mathrm{M}(S) = \operatorname{rk}_\mathrm{M}(E)$.
A \emph{basis} of $\mathrm{M}$ is a subset that is both independent and spanning,
and a \emph{circuit} of $\mathrm{M}$ is a subset that is minimal among those not in any basis of $\mathrm{M}$.

\begin{definition}
Let $\mathrm{M}$ and $\mathrm{N}$ be matroids on ground sets $E$ and $F$.
 A \emph{morphism} $f:\mathrm{M} \to  \mathrm{N}$ is a function from $E$ to $F$ that satisfies
\[
 \operatorname{rk}_{\mathrm{N}}(f(S_2)) -\operatorname{rk}_{\mathrm{N}}(f(S_1)) \le \operatorname{rk}_{\mathrm{M}}(S_2) -\operatorname{rk}_{\mathrm{M}}(S_1)  \ \  \text{for any $S_1\subseteq S_2 \subseteq E$.}
\]
A subset $S \subseteq E$ is a \emph{basis} of $f$ if $S$ is contained in a basis of $\mathrm{M}$ and $f(S)$ contains a basis of $\mathrm{N}$.  
\end{definition}


The category $\mathsf{Mat}$ consists of matroids with morphisms as defined above.\footnote{Morphisms are closely related to  strong maps \cite[Chapter 17]{Wel76}: A strong map from $\mathrm{M}$ to $\mathrm{N}$ is a morphism from $\mathrm{M} \oplus \mathrm{U}_{0,1}$ to $\mathrm{N} \oplus \mathrm{U}_{0,1}$.}  
The initial object of  $\mathsf{Mat}$  is $\mathrm{U}_{0,0}$, the unique matroid on the empty set. 
The terminal object of $\mathsf{Mat}$ is $\mathrm{U}_{0,1}$, the matroid on a singleton of rank zero. 

We write $\mathscr B(f)$ for the set of bases of a morphism of matroids $f$.
When $f$ is the identity morphism of $\mathrm{M}$, the set
$\mathscr B(f)$ is the collection of bases $\mathscr{B}(\mathrm{M})$ of $\mathrm{M}$.
When $f$ is the morphism from $\mathrm{M}$ to the terminal object $\mathrm{U}_{0,1}$, the set $\mathscr B(f)$ is the collection of independent sets $\mathscr{I}(\mathrm{M})$ of $\mathrm{M}$.

We prove the following log-concavity properties for the set of bases of a morphism.
Let $f: \mathrm{M} \to \mathrm{N}$ be any morphism between matroids $\mathrm{M}$ and $\mathrm{N}$ on ground sets $E$ and $F$.
We write $n$ for the cardinality of $E$ and $w=(w_i)_{i \in E}$ for the variables representing the coordinate functions on $\mathbb{R}^E\simeq \mathbb{R}^n$.

\begin{theorem}[Continuous]\label{Continuous}
The basis generating polynomial
\[
\underline{\mathrm{B}}_f(w) \vcentcolon= 
\sum_{S\in \mathscr B(f)}  \ \prod_{i\in S}w_i
\]
is either identically zero or its logarithm is concave on the positive orthant $\mathbb{R}^{n}_{>0}$.
\end{theorem}

\begin{theorem}[Discrete]\label{Discrete}
Let $\mathrm{b}_k(f)$  be the number of bases of $f$ of cardinality $k$. For all $k$, we have
\[
\frac{\mathrm{b}_k(f)^2}{{n \choose k}^2} \geq \frac{\mathrm{b}_{k-1}(f)}{{n \choose k-1}} \frac{\mathrm{b}_{k+1}(f)}{{n \choose k+1}}.
\]
\end{theorem}

One recovers the strongest form of Mason's conjecture on  independent sets of a matroid from \cite{Mas72}, proved in \cite{ALOGV18b,BH18}, by considering the morphism to $\mathrm{U}_{0,1}$ in Theorem \ref{Discrete}. 

\begin{example}
Let $\mathbb{A}^1(\mathbb{F}_2)$ be any line in the four-dimensional affine space $\mathbb{A}^4(\mathbb{F}_2)$ over the field $\mathbb{F}_2$,
and let $\mathrm{M}$ be the matroid of affine dependencies on the 14 points $\mathbb{A}^4(\mathbb{F}_2) \setminus \mathbb{A}^1(\mathbb{F}_2)$.
Projecting from  the line  gives a two-to-one map onto the projective plane 
\[
\mathbb{A}^4(\mathbb{F}_2) \setminus \mathbb{A}^1(\mathbb{F}_2) \longrightarrow \mathbb{P}^2(\mathbb{F}_2).
\]
The projection defines a morphism $f$ from $\mathrm{M}$ to the Fano matroid
$\mathrm{F}_7$ with
\[
\Big(\mathrm{b}_0(f),\mathrm{b}_1(f),\mathrm{b}_2(f),\mathrm{b}_3(f), \mathrm{b}_4(f), \mathrm{b}_5(f),\mathrm{b}_6(f),\ldots\Big) =  \Big(0,0,0,224, 840, 1232,0,\ldots\Big).
\]
See Remark \ref{Remark:Linear} for a discussion of morphisms of matroids constructed from linear maps.
\end{example}

\begin{example}
A graph homomorphism is a function between the vertex sets of two graphs that maps adjacent vertices to adjacent vertices.
For example, consider the following graph homomorphism $G\to H$:
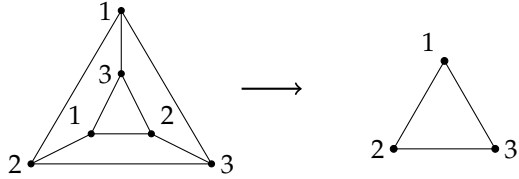
\begin{figure}[h]
\begin{tikzpicture}[scale=0.4]
\draw[fill]
(0,0) circle [radius=0.1] node [left] {2}
(6,0) circle [radius=0.1] node [right] {3}
(3,5.1) circle [radius=0.1] node [left] {1}
(2,1) circle [radius=0.1] node [above left] {1}
(4,1) circle [radius=0.1] node [above right] {2}
(3,3) circle [radius=0.1] node [left] {3}
;
\draw
(0,0) -- (6,0) node [midway, below] {}
(6,0) -- (3,5.1) node [midway, right] {}
(3,5.1) -- (0,0) node [midway, left] {}
(2,1) -- (4,1) node [midway, below] {}
(4,1) -- (3,3) node [midway, right] {}
(3,3) -- (2,1) node [midway, left] {}
(0,0) -- (2,1) node [midway, above] {}
(6,0) -- (4,1) node [midway, above] {}
(3,5.2) -- (3,3) node [midway, below left] {}
;
\draw[thick, ->]
(7,2.5) -- (9,2.5);
\end{tikzpicture}
\qquad
\begin{tikzpicture}[scale=0.45]
\draw[fill]
(0,0) circle [radius=0.1] node [left] {2}
(3,0) circle [radius=0.1] node [right] {3}
(1.5,2.6) circle [radius=0.1] node [above left] {1}
;
\draw
(0,0) -- (3,0) node [midway, below=2mm] {}
(3,0) -- (1.5, 2.6) node [midway, right=3mm] {}
(1.5, 2.6) -- (0,0) node [midway, left=3mm] {}
;
\end{tikzpicture}
\end{figure}

\noindent The induced map between the edges defines a morphism between the cycle matroids $f : \mathrm{M}(G) \to \mathrm{M}(H)$ with
\[
\Big(\mathrm{b}_0(f), \mathrm{b}_1(f), \mathrm{b}_2(f), \mathrm{b}_3(f),  \mathrm{b}_4(f), \mathrm{b}_5(f), \mathrm{b}_6(f),\ldots\Big) = \Big(0,0,27,79,111,75,0,\ldots\Big).
\]
See Remark \ref{Remark:Cycle} for a discussion of morphisms of matroids constructed from graph homomorphisms.
\end{example}

\begin{example}
For any cellularly embedded graph $G$ on a compact surface $\Sigma$, the bijection between the edges of $G$ and its geometric dual $G^*$ on $\Sigma$ defines a morphism of matroids 
\[
f: \mathrm{M}(G)^* \longrightarrow \mathrm{M}(G^*),
\]  
where $\mathrm{M}(G)^*$ is the cocycle matroid of $G$ and $ \mathrm{M}(G^*)$ is the cycle matroid of $G^*$, see \cite[Theorem 4.3]{EMM15} for a proof.
The difference between the ranks of the two matroids is $2-\chi(\Sigma)$,
so $f$ is not an isomorphism unless the surface is  a sphere.
For example, consider the embedding of the complete graph $K_7$ on the torus shown below; it is the 1-skeleton of the minimal triangulation of the torus.  
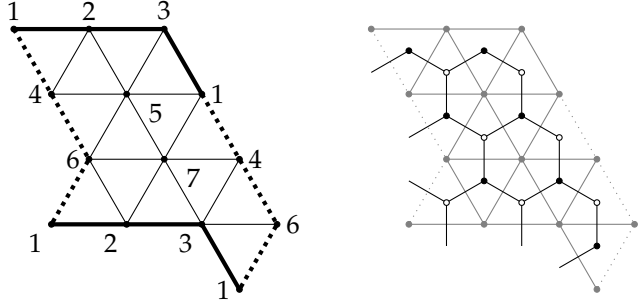
\begin{figure}[h]
\begin{tikzpicture}[scale=0.25]
\newcommand{\myradius}{0.15}
\newcommand{\rtthree}{1.7320508076}
\draw[ultra thick]
(0,0) -- (4,0) -- (8,0) -- (10,-2*\rtthree);
\draw[ultra thick]
(-2,6*\rtthree) -- (2,6*\rtthree) -- (6,6*\rtthree) -- (8,4*\rtthree);
\draw[ultra thick, dotted]
(-2,6*\rtthree) -- (0,4*\rtthree) -- (2,2*\rtthree) -- (0,0);
\draw[ultra thick, dotted]
(8,4*\rtthree) -- (10,2*\rtthree) -- (12,0) -- (10, -2*\rtthree);
\draw[fill]
(0,0) circle [radius=\myradius] node [below left] {1}
(4,0) circle [radius=\myradius] node [below left] {2}
(8,0) circle [radius=\myradius] node [below left] {3}
(10,-2*\rtthree) circle [radius=\myradius] node [left] {1}
(2,2*\rtthree) circle [radius=\myradius] node [left] {6}
(6,2*\rtthree) circle [radius=\myradius] node [below right] {\ \ 7}
(10,2*\rtthree) circle [radius=\myradius] node [right] {4}
(12,0) circle [radius=\myradius] node [right] {6}
(0,4*\rtthree) circle [radius=\myradius] node [left] {4}
(4,4*\rtthree) circle [radius=\myradius] node [below right] {\ \ 5}
(8,4*\rtthree) circle [radius=\myradius] node [right] {1}
(-2, 6*\rtthree) circle [radius=\myradius] node [above] {1}
(2, 6*\rtthree) circle [radius=\myradius] node [above] {2}
(6, 6*\rtthree) circle [radius=\myradius] node [above] {3}
;
\draw
(8,0) -- (12,0)
(2,2*\rtthree) -- (6,2*\rtthree) -- (10,2*\rtthree)
(0,4*\rtthree) -- (4,4*\rtthree) -- (8,4*\rtthree)
(2,6*\rtthree) -- (4,4*\rtthree) -- (6,2*\rtthree) -- (8,0)
(0,4*\rtthree) -- (2, 6*\rtthree)
(2,2*\rtthree) -- (4,4*\rtthree) -- (6,6*\rtthree)
(4,0) -- (6,2*\rtthree) -- (8,4*\rtthree)
(8,0) -- (10,2*\rtthree)
(2,2*\rtthree) -- (4,0)
;
\end{tikzpicture}
\qquad
\begin{tikzpicture}[scale=0.25]
\newcommand{\myradius}{0.15}
\newcommand{\rtthree}{1.7320508076}
\draw[gray]
(0,0) -- (4,0) -- (8,0) -- (10,-2*\rtthree);
\draw[gray]
(-2,6*\rtthree) -- (2,6*\rtthree) -- (6,6*\rtthree) -- (8,4*\rtthree);
\draw[gray,dotted]
(-2,6*\rtthree) -- (0,4*\rtthree) -- (2,2*\rtthree) -- (0,0);
\draw[gray, dotted]
(8,4*\rtthree) -- (10,2*\rtthree) -- (12,0) -- (10, -2*\rtthree);
\draw[fill, gray]
(0,0) circle [radius=\myradius] 
(4,0) circle [radius=\myradius] 
(8,0) circle [radius=\myradius] 
(10,-2*\rtthree) circle [radius=\myradius] node [below] {\ }
(2,2*\rtthree) circle [radius=\myradius] 
(6,2*\rtthree) circle [radius=\myradius] 
(10,2*\rtthree) circle [radius=\myradius] 
(12,0) circle [radius=\myradius] 
(0,4*\rtthree) circle [radius=\myradius] 
(4,4*\rtthree) circle [radius=\myradius] 
(8,4*\rtthree) circle [radius=\myradius] 
(-2, 6*\rtthree) circle [radius=\myradius] 
(2, 6*\rtthree) circle [radius=\myradius] 
(6, 6*\rtthree) circle [radius=\myradius] 
;
\draw[gray]
(8,0) -- (12,0)
(2,2*\rtthree) -- (6,2*\rtthree) -- (10,2*\rtthree)
(0,4*\rtthree) -- (4,4*\rtthree) -- (8,4*\rtthree)
(2,6*\rtthree) -- (4,4*\rtthree) -- (6,2*\rtthree) -- (8,0)
(0,4*\rtthree) -- (2, 6*\rtthree)
(2,2*\rtthree) -- (4,4*\rtthree) -- (6,6*\rtthree)
(4,0) -- (6,2*\rtthree) -- (8,4*\rtthree)
(8,0) -- (10,2*\rtthree)
(2,2*\rtthree) -- (4,0)
;
\draw[fill]
(10, -2*\rtthree/3) circle [radius=\myradius]
(4, 4*\rtthree/3) circle [radius=\myradius]
(8, 4*\rtthree/3) circle [radius=\myradius]
(2, 10*\rtthree/3) circle [radius=\myradius]
(6, 10*\rtthree/3) circle [radius=\myradius]
(0, 16*\rtthree/3) circle [radius=\myradius]
(4, 16*\rtthree/3) circle [radius=\myradius]
;
\draw
(2,2*\rtthree/3) -- (4,4*\rtthree/3) -- (6, 2*\rtthree/3) -- (8,4*\rtthree/3) -- (10,2*\rtthree/3) -- (10,-2*\rtthree/3)
(4, 4*\rtthree/3) -- (4, 8*\rtthree/3) -- (2,10*\rtthree/3) --  (2,14*\rtthree/3) --  (0,16*\rtthree/3)
(8, 4*\rtthree/3) -- (8, 8*\rtthree/3) --(6,10*\rtthree/3) --  (6,14*\rtthree/3) --  (4,16*\rtthree/3) -- (2,14*\rtthree/3)
(4,8*\rtthree/3) -- (6,10*\rtthree/3)
(10,-2*\rtthree/3) -- (8, - 4*\rtthree/3)
(6, 2*\rtthree/3) -- (6, -2*\rtthree/3)
(2, 2*\rtthree/3) -- (2, -2*\rtthree/3)
(2, 2*\rtthree/3) -- (0, 4*\rtthree/3)
(0, 8*\rtthree/3) -- (2,10*\rtthree/3)
(0, 16*\rtthree/3) -- (-2,14*\rtthree/3)
;
\draw[fill=white]
(2, 2*\rtthree/3) circle [radius=\myradius]
(6, 2*\rtthree/3) circle [radius=\myradius]
(10, 2*\rtthree/3) circle [radius=\myradius]
(4, 8*\rtthree/3) circle [radius=\myradius]
(8, 8*\rtthree/3) circle [radius=\myradius]
(2, 14*\rtthree/3) circle [radius=\myradius]
(6, 14*\rtthree/3) circle [radius=\myradius]
;
\end{tikzpicture}
\end{figure}

\noindent The geometric dual $K_7^*$  is the Heawood graph,
the point-line incidence graph of the projective plane  $\mathbb{P}^2(\mathbb{F}_2)$.
\noindent One can check that $f: \mathrm{M}(K_7)^* \to \mathrm{M}(K_7^*)$ is a morphism with
\[
\Big(\ldots,\mathrm{b}_{13}(f), \mathrm{b}_{14}(f), \mathrm{b}_{15}(f),\ldots\Big) = \Big(\ldots,50421,47715,16807,\ldots\Big).
\]
We point to \url{https://github.com/chrisweur/matroidmap} for a Macaulay2 code supporting the computations here.
\end{example}

We identify $E$ with $[n]$, and introduce a variable $w_0$ different from the variables $w_1,\ldots,w_n$.
The \emph{homogeneous multivariate Tutte polynomial} of $f:\mathrm{M} \to \mathrm{N}$ is the homogeneous polynomial of degree $n$ in $n+1$ variables
\[
\mathrm{Z}_{p,q,f}(w_0, w_1, \ldots, w_n) \vcentcolon= \sum_{S\subseteq [n]} p^{-\operatorname{rk}_{\mathrm{M}}(S)}q^{-\operatorname{rk}_{\mathrm{N}}(f(S))} w_0^{n-|S|}\prod_{i\in S} w_i,
\]
where $p$ and $q$ are real  parameters.
 We show that the homogeneous multivariate Tutte polynomials are Lorentzian in the sense of \cite{BH19}, and deduce Theorems \ref{Continuous} and \ref{Discrete} from the Lorentzian property.

\begin{theorem}\label{TheoremLorentzian}
For any positive real numbers $p \le 1$ and $q \le 1$,
the polynomial $\mathrm{Z}_{p,q,f}$ is Lorentzian.
\end{theorem}

One recovers the Lorentzian property of the homogeneous multivariate Tutte polynomial of a matroid $\mathrm{M}$ \cite[Theorem 11.1]{BH19} by considering the morphism from $\mathrm{M}$ to $\mathrm{U}_{0,1}$.
We will establish Theorem \ref{TheoremLorentzian} in the more general context of flag matroids.  

After we review notions surrounding morphisms of matroids in Section \ref{SectionMorphism}, we  turn to flag matroids, which are Coxeter matroids of type A in the sense of \cite{BGW03}, and define homogeneous multivariate Tutte polynomials of flag matroids in Section \ref{SectionTutte}.  We show in Section \ref{SectionLorentzian} that these polynomials are Lorentzian in the sense of \cite{BH19}, and deduce the three main theorems stated above.
We close the paper in Section  \ref{SectionProblems} with a question and a conjecture.

\section{Morphisms of matroids}\label{SectionMorphism}

\subsection{}

Let $\mathrm{M}$ and $\mathrm{N}$ be matroids on ground sets $E$ and $F$.  
A \emph{morphism} $f:\mathrm{M} \to \mathrm{N}$ is a function from $E$ to $F$ satisfying any one of the following equivalent conditions:
\begin{enumerate}[(1)]\itemsep 5pt
\item  For any $S_1 \subseteq S_2 \subseteq E$, we have
$\operatorname{rk}_{\mathrm{N}}(f(S_2)) -\operatorname{rk}_{\mathrm{N}}(f(S_1)) \le \operatorname{rk}_{\mathrm{M}}(S_2) -\operatorname{rk}_{\mathrm{M}}(S_1)$.
\item If $T\subseteq F $ is a cocircuit of $\mathrm{N}$, then $f^{-1}(T) \subseteq E$ is a union of cocircuits of $\mathrm{M}$.
\item If $T\subseteq F $ is a flat of $\mathrm{N}$, then $f^{-1}(T)\subseteq E$ is a flat of $\mathrm{M}$.
\end{enumerate}
For all undefined terms in matroid theory, we refer to \cite{Oxl11}.  
The equivalence of the three conditions readily follows from the properties of strong maps \cite{Kun86}.  
Basic properties of the category $\mathsf{Mat}$ with morphisms as defined above was  studied in \cite{HP18}.

\begin{remark}[Representable matroids]\label{Remark:Linear}
Let $\mathsf{Mat}(\mathbb{F})$ be the category whose objects are functions of the form 
\[
\varphi:E \longrightarrow W,
\] 
where $E$ is a finite set and $W$ is a vector space over a field $\mathbb{F}$.
We write $\mathrm{M}(\varphi)$ for the matroid on $E$ defined by the rank function
\[
\textrm{rk}_{\mathrm{M}(\varphi)}(S)= \text{the dimension of the span of $\varphi(S)$ over $\mathbb{F}$}.
\]
The object $\varphi$ is called a \emph{representation} of $\mathrm{M}(\varphi)$ over $\mathbb{F}$ \cite[Chapter 6]{Wel76}.
A morphism from $\varphi_1$ to $\varphi_2$ in $\mathsf{Mat}(\mathbb F)$  is a commutative diagram
\[
\xymatrix{
E_1 \ar[r]^{\varphi_1} \ar[d]& W_1 \ar[d]\\
E_2 \ar[r]^{\varphi_2}& W_2
}
\]
where  $W_1 \to W_2$ is a linear map between the vector spaces.
Using the description of morphisms of matroids in terms of flats, we see that the function $E_1 \to E_2$ gives a morphism of matroids $\mathrm{M}(\varphi_1) \to \mathrm{M}(\varphi_2)$,
defining a functor   
\[
\mathscr{R}_\mathbb{F}:\mathsf{Mat}(\mathbb{F}) \longrightarrow \mathsf{Mat}, \quad \varphi \longmapsto \mathrm{M}(\varphi).
\]
\end{remark}

\begin{remark}[Cycle matroids]\label{Remark:Cycle}
Let $\mathsf{Graph}$ be the category of graphs, that is, functions  of the form 
\[
\varphi:E  \longrightarrow V^{(2)},
\]
 where $E$ is a finite set  and $V^{(2)}$ is the set of two-element multi-subsets of another finite set $V$.
We write $\mathrm{M}(\varphi)$ for the matroid on $E$ defined by the condition
\[
\text{$S$ is independent in $\mathrm{M}(\varphi)$} \Longleftrightarrow \text{$S$ does not contain any cycle of the graph $\varphi$.}
\]
The matroid $\mathrm{M}(\varphi)$ is called the \emph{cycle matroid} of  $\varphi$ \cite[Chapter 5]{Oxl11}.
A morphism from $\varphi_1$
to $\varphi_2$ in $\mathsf{Graph}$ is a commutative diagram
\[
\xymatrix{
E_1 \ar[r]^{\varphi_1} \ar[d]&  V_1^{(2)} \ar[d]\\
E_2  \ar[r]^{\varphi_2} &  V_2^{(2)}
}
\]
where $V_1^{(2)} \to V_2^{(2)}$ is a map induced from a map $V_1 \to V_2$.
Using the description of morphisms of matroids in terms of flats, we see that  the function $E_1 \to E_2$ gives a morphism of matroids $\mathrm{M}(\varphi_1) \to \mathrm{M}(\varphi_2)$,
defining a functor
\[
\mathscr{C}:\mathsf{Graph} \longrightarrow \mathsf{Mat}, \quad  \varphi \longmapsto \mathrm{M}(\varphi).
\]
\end{remark}

\subsection{}

 A \emph{matroid quotient} is a morphism of matroids $f: \mathrm{M} \to \mathrm{N}$ whose underlying map between the ground sets is the identity function of a finite set.
 In this case, $\mathrm{N}$ is said to be a quotient of $\mathrm{M}$,
 and we denote the morphism by $\mathrm{M} \twoheadrightarrow \mathrm{N}$.
 Many equivalent descriptions of matroid quotients are given in \cite[Proposition 7.4.7]{Bry86}.  
 For later use, we record here two immediate but useful properties of matroid quotients.
Recall that an element $i$ is a \emph{loop} of $\mathrm{M}$ if $\{i\}$ is a circuit of $\mathrm{M}$, and that distinct elements $i$, $j$ are \emph{parallel} in $\mathrm{M}$ if $\{i,j\}$ is a circuit of $\mathrm{M}$.

\begin{lemma}\label{lem:matquot}
Let $\mathrm{M} \twoheadrightarrow \mathrm{N}$ be a matroid quotient on $E$, and let $i$, $j$ be  distinct elements of  $E$.
\begin{enumerate}[(1)]\itemsep 5pt
\item If $i$ is a loop in $\mathrm{M}$, then $i$ is a loop in $\mathrm{N}$. 
\item  If $i$, $j$ are parallel in $\mathrm{M}$, then either $i$, $j$ are parallel in $\mathrm{N}$ or both $i$, $j$ are loops in $\mathrm{N}$.
\end{enumerate}
\end{lemma}

\begin{proof}
The first statement follows from $ \operatorname{rk}_\mathrm{N}(i)  \le \operatorname{rk}_\mathrm{M}(i)$.
For the second statement, note that
\[
\operatorname{rk}_\mathrm{N}(i, j) - \operatorname{rk}_\mathrm{N}(i) \le \operatorname{rk}_\mathrm{M}(i, j) - \operatorname{rk}_\mathrm{M}(i) \ \ \text{and} \ \
\operatorname{rk}_\mathrm{N}(i, j) - \operatorname{rk}_\mathrm{N}(j) \le \operatorname{rk}_\mathrm{M}(i, j) - \operatorname{rk}_\mathrm{M}(j).
\]
Thus $\operatorname{rk}_\mathrm{M}(i, j) = \operatorname{rk}_\mathrm{M}(i)  = \operatorname{rk}_\mathrm{M}(j)$
implies  $\operatorname{rk}_\mathrm{N}(i, j) = \operatorname{rk}_\mathrm{N}(i)  = \operatorname{rk}_\mathrm{N}(j)$.
\end{proof}

Questions on morphisms of matroids can often be reduced to those on matroid quotients.
Let $\mathrm{M}$ and $\mathrm{N}$ be matroids on ground sets $E$ and $F$,
and let $f$ be a function from $E$ to $F$.
We write $f^{-1}(\mathrm{N})$ for the matroid on $E$ defined by the rank function
\[
 \operatorname{rk}_{f^{-1}(\mathrm{N})}(S)  = \operatorname{rk}_\mathrm{N}(f(S)) \ \ \text{for $S \subseteq E$}.
\]
Informally, $f^{-1}(\mathrm{N})$ is the matroid obtained from  the restriction $\mathrm{N}|_{f(E)}$ by replacing each non-loop element $e$ with a set of parallel elements $f^{-1}(e)$
and each loop $e$ by a set of loops $f^{-1}(e)$. 
See \cite[Chapter 8.2]{Wel76} for a more general construction of \emph{induced matroids}.

\begin{lemma}
The function $f$ defines a morphism $\mathrm{M} \to \mathrm{N}$ if and only if $f^{-1}(\mathrm{N})$ is a quotient of $\mathrm{M}$.
In this case, the set of bases $\mathscr{B}(\mathrm{M} \to \mathrm{N})$ is either empty or equal to  $\mathscr{B}(\mathrm{M} \twoheadrightarrow f^{-1}(\mathrm{N}))$.
\end{lemma}

Therefore, the basis generating polynomial of a morphism $f:\mathrm{M} \to \mathrm{N}$ is either identically zero or equal to the basis generating polynomial of the quotient $\mathrm{M} \twoheadrightarrow f^{-1}(\mathrm{N})$.

\begin{proof}
The first statement is obvious, given that  $f^{-1}(\mathrm{N})$ is a matroid.
For the second statement, note that  $\mathscr{B}(\mathrm{M} \to \mathrm{N})$ is nonempty if and only if $f(E)$ is a spanning set of $\mathrm{N}$.
In this case, $S \subseteq E$ is a spanning set of $f^{-1}(\mathrm{N})$ if and only if
$f(S) \subseteq F$ is a spanning set of $\mathrm{N}$, and hence
\[
\mathscr{B}(\mathrm{M} \to \mathrm{N})=\mathscr{B}(\mathrm{M} \twoheadrightarrow f^{-1}(\mathrm{N})).
\qedhere
\]
\end{proof}

We remark that the collection of bases of a quotient is the collection of feasible sets of a \emph{saturated delta-matroid}  and conversely \cite{Tar85}.
Such a nonempty collection $\mathscr{F}$ is characterized by its properties
\begin{enumerate}[(1)]\itemsep 5pt
\item for any $S_1,S_2 \in \mathscr{F}$ and any $S_3$ containing $S_1$ and contained in $S_2$, we have $S_3 \in \mathscr{F}$, and
\item for any $S_1,S_2 \in \mathscr{F}$ and any $i \in S_1 \triangle S_2$, there is $j \in S_1 \triangle S_2$ such that $S_1 \triangle \{i,j\} \in \mathscr{F}$.
\end{enumerate}
The collection of bases of $\mathrm{M} \twoheadrightarrow \mathrm{N}$ of a given cardinality $k$ is, when nonempty, the set of bases of a matroid, the rank $k$ \emph{Higgs lift} of $\mathrm{N}$ toward $\mathrm{M}$ \cite[Exercise 7.20]{Bry86}.

\section{The multivariate Tutte polynomial of a flag matroid}\label{SectionTutte}

\subsection{}

Let $\mathrm{M}$ be a matroid on  $E$, and let $w=(w_i)_{i \in E}$.
The \emph{multivariate Tutte polynomial}, also called the Potts model partition function, of $\mathrm{M}$ is the polynomial
\[
\underline{\mathrm{Z}}_{q,\mathrm{M}}(w) \vcentcolon= \sum_{S\subseteq E} q^{-\operatorname{rk}_\mathrm{M}(S)} \prod_{i\in S}w_i,
\]
where $q$ is a real parameter \cite{Sok05}.
The polynomial satisfies the \emph{deletion-contraction relation}
\[
\underline{\mathrm{Z}}_{q,\mathrm{M}} = \underline{\mathrm{Z}}_{q,\mathrm{M} \setminus i}+ q^{-\operatorname{rk}_\mathrm{M}(i)} \ w_i\  \underline{\mathrm{Z}}_{q,\mathrm{M}/i} \ \ \text{for any $i \in E$}.
\]
It  is related to the usual Tutte polynomial \cite[Chapter 15]{Wel76} , denoted $\mathrm{T}_\mathrm{M}(x,y)$, by the change of variables
\[
q=(x-1)(y-1)\ \ \text{and} \ \ w_i= (y-1) \ \ \text{for all $i \in E$}.
\]
More precisely, with the above values of $q$ and $w$, we have 
\[
(x-1)^{-\operatorname{rk}_\mathrm{M}(E)}\ \mathrm{T}_\mathrm{M}(x,y) = \underline{\mathrm{Z}}_{q,\mathrm{M}}(w).
\]
We refer to \cite{Sok05} for more combinatorial properties of the multivariate Tutte polynomial and its connection to statistical physics.
Two notable limits are
\[
\lim_{q\to 0} \ q^{\operatorname{rk}_\mathrm{M}(E)}  \ \underline{\mathrm{Z}}_{q,\mathrm{M}}(w)  = \sum_{S \in \mathscr{S}(\mathrm{M})} \prod_{i \in S} w_i \ \  \text{and}\ \ 
\lim_{q\to 0} \  \underline{\mathrm{Z}}_{q,\mathrm{M}}(qw)  = \sum_{S \in \mathscr{I}(\mathrm{M})} \prod_{i \in S} w_i,
\]
where $ \mathscr{S}(\mathrm{M})$ is the collection of spanning sets of $\mathrm{M}$
and $ \mathscr{I}(\mathrm{M})$ is the collection of independent sets of $\mathrm{M}$.

In \cite{LV75}, Las Vergnas introduced Tutte polynomials of matroid quotients and showed that the bases of a matroid quotient serve the same role as the bases of a matroid in defining the Tutte polynomial via internal-external activities.  
We refer to the series of papers \cite{LV80,TutteII,TutteI,TutteIII,TutteIV,TutteV} for properties and applications.

\begin{definition}
The \emph{Tutte polynomial  of a matroid quotient} $\mathrm{M} \twoheadrightarrow \mathrm{N}$ on $E$ is the trivariate polynomial
\[
\mathrm{T}_{\mathrm{M}\twoheadrightarrow \mathrm{N}}(x,y,z) \vcentcolon= \sum_{S\subseteq E} (x-1)^{\operatorname{crk}_{\mathrm{N}}(S)}(y-1)^{|S|-\operatorname{rk}_\mathrm{M}(S)}z^{\operatorname{crk}_\mathrm{M}(S) - \operatorname{crk}_{\mathrm{N}}(S)},
\]
where $\operatorname{crk}_{\mathrm{M}}(S)=\operatorname{rk}_{\mathrm{M}}(E)-\operatorname{rk}_{\mathrm{M}}(S)$ and $\operatorname{crk}_{\mathrm{N}}(S)=\operatorname{rk}_{\mathrm{N}}(E)-\operatorname{rk}_{\mathrm{N}}(S)$.
\end{definition}

The usual Tutte polynomial of a matroid  is recovered by setting $\mathrm{M} = \mathrm{N}$. 
We define the \emph{multivariate Tutte polynomial} of $\mathrm{M} \twoheadrightarrow \mathrm{N}$ by
\[
\underline{\mathrm{Z}}_{p,q,\mathrm{M} \twoheadrightarrow \mathrm{N}}(w)  \vcentcolon= \sum_{S\subseteq E} p^{-\operatorname{rk}_{\mathrm{M}}(S)}q^{-\operatorname{rk}_{\mathrm{N}}(S)} \prod_{i\in S} w_i,
\]
where $p$ and $q$ are real parameters.
 The Tutte polynomial of $\mathrm{M} \twoheadrightarrow \mathrm{N}$  and the multivariate Tutte polynomial of  $\mathrm{M} \twoheadrightarrow \mathrm{N}$ are related by the change of variables
\[
p = z(y-1), \ \ pq = (x-1)(y-1) \ \ \text{and} \ \ w_i = (y-1) \ \  \text{for all $i \in E$}.
\]
More precisely, with the above values of $p$, $q$ and $w$, we have 
\[
(x-1)^{-\operatorname{rk}_\mathrm{N}(E)} z^{\operatorname{rk}_\mathrm{N}(E)-\operatorname{rk}_\mathrm{M}(E)} \ \mathrm{T}_{\mathrm{M}\twoheadrightarrow \mathrm{N}}(x,y,z) = \underline{\mathrm{Z}}_{p,q, \mathrm{M} \twoheadrightarrow \mathrm{N}}(w).
\]
Theorem \ref{TheoremLorentzian} implies that the multivariate Tutte polynomial 
is log-concave on the positive orthant $\mathbb{R}^E_{>0}$  for any positive real numbers $p \le 1$ and $q \le 1$.

\subsection{}

According to \cite{BGW03}, flag matroids are precisely the Coxeter matroids of type A.
For our purposes, it is most natural to work with flag matroids and their multivariate Tutte polynomials.

\begin{definition}\label{FlagMatroids}
A \emph{flag matroid} $\mathscr M$ is a sequence of matroids $(\mathrm{M}_1, \ldots, \mathrm{M}_\ell)$ of matroids on a common ground set $E$ satisfying the condition
\[
\text{the matroid $\mathrm{M}_k$ is a quotient of $\mathrm{M}_{k+1}$ for all $0<k<\ell$.}
\]
 The matroids $\mathrm{M}_1, \ldots, \mathrm{M}_\ell$ are \emph{constituents} of the flag matroid $\mathscr  M$.
\end{definition}

Our definition of flag matroids, which agrees with \cite[Definition 6.2]{CDMS18}, differs slightly from the one in  \cite[Section 1.7]{BGW03}
 in that we allow $\mathrm{M}_k=\mathrm{M}_{k+1}$.
It is necessary to work with Definition \ref{FlagMatroids} to construct deletion $\mathscr{M} \setminus i$ and contraction $\mathscr{M} / i$   by respective operations on the constituents of $\mathscr{M}$.

\begin{definition}
Let $\mathscr{M}=(\mathrm{M}_1, \ldots, \mathrm{M}_\ell)$ be a flag matroid on $E$, and let
 $i$ be any element of $E$.
\begin{enumerate}[(1)]\itemsep 5pt
\item The \emph{deletion} $\mathscr M \setminus i $ of $\mathscr M$ is the flag matroid
\[
\mathscr M \setminus i = (\mathrm{M}_1\setminus i , \ldots, \mathrm{M}_\ell \setminus i)
\]
where $\mathrm{M}_k\setminus i$ is the matroid on $E \setminus i$ defined by the rank function
\[
\operatorname{rk}_{\mathrm{M}_k\setminus i}(S)=\operatorname{rk}_{\mathrm{M}_k}(S). 
\]
\item The \emph{contraction} $\mathscr M / i $ of $\mathscr M$ is the flag matroid
\[
 \mathscr{M} / i = (\mathrm{M}_1 / i , \ldots, \mathrm{M}_\ell / i),
\]
where  $\mathrm{M}_k/ i$ is the matroid on $E \setminus i$ defined by the rank function
\[
\operatorname{rk}_{\mathrm{M}_k/ i}(S)=\operatorname{rk}_{\mathrm{M}_k}(S \cup i)-\operatorname{rk}_{\mathrm{M}_k}(i).
\]
\end{enumerate}
\end{definition}

It is straightforward to check that $\mathscr M \setminus i $ and $\mathscr M / i$  are flag matroids on $E \setminus i$. 
For the rest of this paper, we identify the ground set $E$ with $[n]$
and write $w_0$ for a variable different from the variables $w_1,\ldots,w_n$.

\begin{definition}\label{defn:tutte}
The \emph{homogenous multivariate Tutte polynomial of a flag matroid} $\mathscr{M}=(\mathrm{M}_1, \ldots, \mathrm{M}_\ell)$ is the homogenous polynomial of degree $n$ in $n+1$ variables
\[
\mathrm{Z}_{q,\mathscr{M}}(w_0,w_1,\ldots,w_n) \vcentcolon= \sum_{S \subseteq [n]} q_1^{-\operatorname{rk}_{\mathrm{M}_1}(S)}  q_2^{-\operatorname{rk}_{\mathrm{M}_2}(S)} \cdots q_\ell^{-\operatorname{rk}_{\mathrm{M}_\ell}(S)}  w_0^{n-|S|}\prod_{i \in S} w_i,
\]
where $q$ stands for the sequence of  real parameters $(q_1,\ldots,q_\ell)$.
\end{definition}

When $\ell=2$ and $w_0=1$, we recover the multivariate Tutte polynomial of a matroid quotient.
 In general, the homogeneous multivariate Tutte polynomial satisfies the \emph{deletion-contraction relation}
\[
\mathrm{Z}_{q,\mathscr  M} = w_0\ \mathrm{Z}_{q,\mathscr M\setminus i } + w_i \ q_1^{-\operatorname{rk}_{\mathrm{M}_1}(i)} q_2^{-\operatorname{rk}_{\mathrm{M}_2}(i)} \cdots q_\ell^{-\operatorname{rk}_{\mathrm{M}_\ell}(i)} \mathrm{Z}_{q,\mathscr M/i} \ \ \text{for any $i \in E$.}
\]
From the deletion-contraction relation, we see that the $i$-th partial derivative  of $\mathrm{Z}_{q,\mathscr  M}$ is 
\[
\frac{\partial}{\partial w_i} \mathrm{Z}_{q,\mathscr  M}=q_1^{-\operatorname{rk}_{\mathrm{M}_1}(i)} q_2^{-\operatorname{rk}_{\mathrm{M}_2}(i)} \cdots q_\ell^{-\operatorname{rk}_{\mathrm{M}_\ell}(i)} \mathrm{Z}_{q,\mathscr M/i}  \ \ \text{for any $i \in E$.}
\]
The above formula 
will play a central role in the inductive proof of 
Theorem \ref{MainTheorem}.

\begin{remark}
The homogeneous multivariate Tutte polynomials
are the \emph{reduced multivariate Tutte characters}
for the minor system of flag matroids with $\ell$ constituents in the sense of \cite{DFM18}.
See \cite{KMT18} for an equivalent theory of \emph{canonical Tutte polynomials} of  minor systems.
\end{remark}

\section{The Lorentzian property}\label{SectionLorentzian}

\subsection{}

In \cite{BH19}, the authors introduce Lorentzian polynomials as a generalization of volume polynomials in algebraic geometry and stable polynomials in optimization theory.  
These polynomials capture the essence of many log-concavity phenomena in combinatorics.   Here we briefly summarize the relevant results.
We write $e_i$ for the $i$-th standard unit vector of $\mathbb{N}^n$ and $\partial_i$ for the differential operator $\frac{\partial}{\partial w_i}$
on the polynomial ring $\mathbb{R}[w_1,\ldots,w_n]$.

\begin{definition}\cite[Definition 2.1]{BH19}
A homogeneous polynomial $h \in \mathbb{R}[w_1,\ldots,w_n]$ of degree $d$ is \emph{strictly Lorentzian} if all of its coefficients are positive and,
for any indices $i_1,\ldots,i_{d-2} \in [n]$, the quadratic form $\partial_{i_1} \cdots \partial_{i_{d-2}}h$ has the Lorentzian signature $(+,-,\ldots,-)$.
\emph{Lorentzian polynomials} are polynomials that can be obtained as a limit of strictly Lorentzian polynomials.
\end{definition}

A subset $\mathrm{J} \subseteq \mathbb{N}^n$ is \emph{$\mathrm{M}$-convex} if, for any index $i \in [n]$ and any vectors $\alpha \in \mathrm{J}$ and $\beta\in \mathrm{J}$ whose $i$-th coordinates satisfy $\alpha_i>\beta_i$, there is an index $j \in [n]$ satisfying
\[
\alpha_j<\beta_j \ \ \text{and} \ \ \alpha-e_i+e_j \in \mathrm{J} \ \ \text{and} \ \ \beta-e_j+e_i \in \mathrm{J}.
\]
The notion of $\mathrm{M}$-convexity forms the basis of discrete convex analysis \cite{Mur03}.
The \emph{support} of $h$ is the set of monomials appearing in $h$, viewed as a subset of $\mathbb{N}^n$.

\begin{theorem}\cite[Theorem 5.1]{BH19}\label{TheoremCharacterization}
A homogeneous polynomial $h \in \mathbb{R}[w_1,\ldots,w_n]$ of degree $d$ is Lorentzian if and only if all of its coefficients are nonnegative, its support is $\mathrm{M}$-convex, and, for any indices 
$i_1,\ldots,i_{d-2} \in [n]$, the quadratic form $\partial_{i_1} \cdots \partial_{i_{d-2}}h$ has at most one positive eigenvalue.
\end{theorem}

For example, a bivariate polynomial $\sum_{k=0}^d a_k w_1^kw_2^{d-k}$ with nonnegative coefficients is Lorentzian if and only if the sequence $a_0,\ldots,a_d$ has no internal zeros\footnote{The sequence $a_0,\ldots,a_d$ has \emph{no internal zeros} if $a_{k_1}a_{k_3} \neq 0 \Longrightarrow a_{k_2} \neq 0$ for all $0 \le k_1 <k_2<k_3\le d$.}
and, for all $0<k<d$, we have
\[
\frac{a_k^2}{{d \choose k}^2} \ge \frac{a_{k-1}}{{d \choose k-1}}\frac{a_{k+1}}{{d \choose k+1}}.
\]

Applications to log-concavity phenomena in combinatorics arise from the following properties of Lorentzian polynomials.
Following \cite{Gur09}, we say that a polynomial $h \in \mathbb{R}[w_1,\ldots,w_n]$  with nonnegative coefficients is \emph{strongly log-concave} if,
for any sequence of indices $i_1,i_2,\ldots  \in [n]$ and any positive integer $k$,
the functions $h$ and  $\partial_{i_1} \cdots \partial_{i_k} h$ are either identically zero or log-concave on the positive orthant $\mathbb{R}^n_{>0}$.

\begin{theorem}\label{TheoremProperties}
Let $h$ and $g$ be homogeneous polynomials in $\mathbb{R}[w_1,\ldots,w_n]$ with nonnegative coefficients.
\begin{enumerate}[(1)]\itemsep 5pt
\item \cite[Theorem 5.3]{BH19} The polynomial $h$   is Lorentzian if and only if $h$ is strongly log-concave.
\item \cite[Theorem 2.10]{BH19} If $h(w)$ is Lorentzian, then $h(Av)$ is Lorentzian for any vector of variables $v=(v_1,\ldots,v_m)$ and any $n \times m$ matrix $A$ with nonnegative entries.
\item \cite[Corollary 5.5]{BH19} If $h$ and $g$ are Lorentzian, then the product $hg$ is Lorentzian.
\end{enumerate}
\end{theorem}

\subsection{}

We now prove the main result of this paper.
Let $\mathscr{M}=(\mathrm{M}_1,\ldots,\mathrm{M}_\ell)$ be a flag matroid on the ground set $[n]$ with  $n \ge 2$ and $\ell \ge 1$.

\begin{theorem}\label{MainTheorem}
The homogeneous multivariate Tutte polynomial $\mathrm{Z}_{q,\mathscr M}(w_0,w_1,\ldots,w_n)$  is Lorentzian for any positive real numbers $q_1,\ldots,q_\ell \leq 1$.
\end{theorem}

For any element $k$ of $[\ell]$ and distinct elements $i$, $j$ of $[n]$, we set
\[
d_k(i,j)\vcentcolon=\operatorname{rk}_{\mathrm{M}_k}(i)+\operatorname{rk}_{\mathrm{M}_k}(j)-\operatorname{rk}_{\mathrm{M}_k}(i,j)=
 \left\{
\begin{array}{ll}
1 & \text{if $i$ and $j$ are parallel in $\mathrm{M}_k$,}\\
0 & \text{if $i$ and $j$ are not parallel in $\mathrm{M}_k$.}
\end{array}\right. 
\]
In other words, $d_k$ is the indicator function for the two-element circuits of $\mathrm{M}_k$.
We write $\mathrm{P}_\mathscr{M}$ for the multi-affine polynomial
\[
\mathrm{P}_{\mathscr M}(q,w)
\vcentcolon=
\sum_{1\leq i< j \leq n} 
q_1^{
d_{1}(i,j)
}
q_2^{
d_{2}(i,j)
}
\cdots 
q_\ell^{
d_{\ell}(i,j)
}
w_iw_j,
\]
where
$q=(q_1,\ldots,q_\ell)$ and
$w=(w_1,\ldots,w_n)$.
Note that  $\mathrm{P}_\mathscr{M}$  depends only on the rank $2$ truncations of the constituents of $\mathscr{M}$.

\begin{lemma}\label{MainLemma}
For any real numbers $w_1,\ldots,w_n$ and any nonnegative real numbers $q_1,\ldots,q_\ell \le 1$,
\[
\frac{1}{2}\Big(1-\frac{1}{n}\Big) (w_1+\cdots+ w_n)^2 \ge   \mathrm{P}_{\mathscr M}(q,w).
\]
\end{lemma}

\begin{proof}
We prove the statement by induction on $(n,\ell)$. 
The case $(2,\ell)$ is straightforward:
\[
\frac{1}{4}(w_1+w_2)^2 \ge \frac{1}{4}q_1^{d_1(1,2)}\ldots q_\ell^{d_\ell(1,2)}(w_1+w_2)^2  \ge q_1^{d_1(1,2)}\ldots q_\ell^{d_\ell(1,2)} w_1w_2.
\]
Since  $\mathrm{P}_\mathscr{M}$ is a linear in the parameter $q_\ell$,
we may suppose that $q_\ell$ is $0$ or $1$.\footnote{Of course, in this context, $0^0=1$.}
Therefore, the following five special cases imply the general case:
\begin{enumerate}[(1)]\itemsep 5pt
\item when $q_\ell=1$ and $\ell=1$; 
\item when $q_\ell=1$ and $\ell>1$; 
\item when $q_\ell=0$ and $\mathrm{M}_\ell$ has a pair of parallel elements, say $1$ and $2$;
\item when $\mathrm{M}_\ell$ has no pair of parallel elements and $\ell=1$;
\item when $\mathrm{M}_\ell$ has no pair of parallel elements and $\ell>1$.
\end{enumerate}

In cases (1) and (4), we need to show
\[
\frac{1}{2}\Big(1-\frac{1}{n}\Big)  (w_1+\cdots+ w_n)^2 \ge  \sum_{1 \le i<j \le n} w_iw_j. 
\]
The above displayed inequality is equivalent to the statement
\[
n(w_1^2+\cdots+ w_n^2) \ge (w_1+\cdots+w_n)^2,
\]
which is the Cauchy-Schwarz inequality for the vectors $(1, . . . , 1)$ and $(w_1, . . . , w_n)$.

In cases (2) and (5), we have
\[
\mathrm{P}_{\mathscr M}(q,w)=\mathrm{P}_{\mathscr N}(q_1,\ldots,q_{\ell-1},w_1,\ldots,w_n),
\]
where $\mathscr{N}$ is the flag matroid with constituents $\mathrm{M}_1,\ldots,\mathrm{M}_{\ell-1}$.
We use  induction on $\ell$.

In case (3), where  $d_\ell(1,2)=1$, 
Lemma \ref{lem:matquot} implies that
\[
d_k(1,i)=d_k(2,i) \ \ \text{for all $i \neq 1,2$ and all $k\neq \ell$}.
\]
Since $q_\ell=0$, the support of $\mathrm{P}_\mathrm{M}$ does not contain $w_1w_2$,
and hence the above shows
\[
\mathrm{P}_\mathscr{M}(q,w)
=\mathrm{P}_{\mathscr{M}\setminus 1}(q_1,\ldots,q_{\ell},w_1+w_2,w_3,\ldots,w_n).
\]
Therefore, by induction on $n$, we have
\[
\frac{1}{2}\Big(1-\frac{1}{n}\Big)  (w_1+\cdots+ w_n)^2 \ge \frac{1}{2}\Big(1-\frac{1}{n-1}\Big)   (w_1+\cdots+ w_n)^2\ge \mathrm{P}_\mathscr{M}(q,w).
\qedhere
\]
\end{proof}

\begin{proof}[Proof of Theorem \ref{MainTheorem}]
We prove the statement by induction on $n$, using Theorem \ref{TheoremCharacterization}.
It is straightforward to check directly that the support of the homogeneous multivariate Tutte polynomial is $\mathrm{M}$-convex whenever $q_1,\ldots,q_\ell$ are positive.\footnote{Alternatively, we may use that $\mathrm{Z}_{\mathbf{1},\mathscr M}=\prod_{i=1}^n (w_0+w_i)$ is a Lorentzian polynomial by Theorem \ref{TheoremProperties} (3).}


As before, we write $w=(w_1,\ldots,w_n)$. 
We first show that the quadratic form
\[
\frac{\partial^{n-2}}{\partial w_0^{n-2}}
\mathrm{Z}_{q,\mathscr M}=\frac{n!}{2}w_0^2+(n-1)!\hspace{0.5mm} \mathrm{Z}^{(1)}_{q,\mathscr M}(w) \ w_0+(n-2)!\hspace{0.5mm} \mathrm{Z}^{(2)}_{q,\mathscr M}(w)
\]
has at most one positive eigenvalue for positive parameters $q_1,\ldots,q_\ell \le 1$.
For this, it suffices to show that the Schur complement of the first principal minor is negative semidefinite. In other words,
the discriminant of the displayed quadratic form with respect to $w_0$ is nonnegative:
\[
\frac{1}{2} \Big(1-\frac{1}{n}\Big)\mathrm{Z}_{q,\mathscr M}^{(1)}(w)^2 \geq \mathrm{Z}_{q,\mathscr M}^{(2)}(w) \ \  \text{for all $w\in \mathbb{R}^n$}.
\]
We prove the discriminant inequality after making the invertible change of variables
\[
w_i \longmapsto q_1^{\operatorname{rk}_{\mathrm{M}_1}(i)} \cdots q_\ell^{\operatorname{rk}_{\mathrm{M}_\ell}(i)} w_i \ \ \text{for all $i \in [n]$.}
\]
The inequality then becomes that of Lemma \ref{MainLemma}:
\[
\frac{1}{2}\Big(1-\frac{1}{n}\Big) (w_1+\cdots+ w_n)^2 \ge   \mathrm{P}_{\mathscr M}(q,w).
\]

It is now enough to show that the $i$-th partial derivative of the homogeneous multivariate Tutte polynomial is Lorentzian for any $i \in [n]$  and any positive $q_1,\ldots,q_\ell\le 1$.
This follows from the induction hypothesis on $n$ and the identity
\[
\partial_i \mathrm{Z}_{q,\mathscr  M}=q_1^{-\operatorname{rk}_{\mathrm{M}_1}(i)} q_2^{-\operatorname{rk}_{\mathrm{M}_2}(i)} \cdots q_\ell^{-\operatorname{rk}_{\mathrm{M}_\ell}(i)} \mathrm{Z}_{q,\mathscr M/i}. 
\qedhere
\]
\end{proof}

Notice that the proof of Theorem \ref{MainTheorem} entirely consists of, apart from the routine induction on $n$,
analysis of rank $2$ matroids in Lemma \ref{MainLemma}.

\begin{proof}[Proof of Theorem \ref{TheoremLorentzian}]
Apply Theorem \ref{MainTheorem} to the flag matroid $(f^{-1}(\mathrm{N}),\mathrm{M})$.
\end{proof}

Let $\mathrm{M}$ be a matroid on $[n]$, and let $\mathrm{N}$ be a matroid on $[m]$.

\begin{corollary}\label{CorollaryBasisGenerating}
The \emph{homogeneous basis generating polynomial}
\[
\mathrm{B}_f(w_0,w_1,\ldots,w_n)\vcentcolon=\sum_{S \in \mathscr{B}(f)} w_0^{n-|S|} \prod_{i \in S} w_i
\]
is a Lorentzian polynomial
for any morphism of matroids  $f:\mathrm{M} \to \mathrm{N}$.
\end{corollary}

One recovers the Lorentzian property of the basis generating polynomial of $\mathrm{M}$ \cite[Section 7]{BH19} from the case $f=\text{id}_\mathrm{M}$ by taking the partial derivative $\big(\frac{\partial}{\partial w_0}\big)^{n-\operatorname{rk}_\mathrm{M}[n]}$.
One recovers the Lorentzian property of the homogeneous independent set generating polynomial of $\mathrm{M}$ \cite[Section 11]{BH19} from the case $\mathrm{N}=\mathrm{U}_{0,1}$.\footnote{These important special cases  were obtained independently in \cite{AOGV18,ALOGV18b}. See \cite{AOGV18,ALOGV18a} for algorithmic applications.}

\begin{proof}
By Theorem \ref{TheoremLorentzian},
the homogeneous multivariate Tutte polynomial
\[
\mathrm{Z}_{p,q,f}(w_0, w_1, \ldots, w_n) = \sum_{S\subseteq [n]} p^{-\operatorname{rk}_{\mathrm{M}}(S)}q^{-\operatorname{rk}_{\mathrm{N}}(f(S))} w_0^{n-|S|}\prod_{i\in S} w_i,
\]
is a Lorentzian polynomial for any positive real numbers $p \le 1$ and $q \le 1$.
Therefore, the limit
\[
\lim_{p \to 0} \lim_{q \to 0} \ q^{\operatorname{rk}_\mathrm{N} f[n]} \ \mathrm{Z}_{p,q,f}(w_0, pw_1, \ldots, pw_n) =\sum_{S \in \mathscr{B}(f)} w_0^{n-|S|} \prod_{i \in S} w_i
\]
is a Lorentzian polynomial.
\end{proof}

\begin{proof}[Proofs of Theorems \ref{Continuous} and \ref{Discrete}]
Theorem \ref{TheoremProperties} (1) and Corollary \ref{CorollaryBasisGenerating} show that
\[
\text{the polynomial $\mathrm{B}_f(1,w_1,\ldots,w_n)=\underline{\mathrm{B}}_f(w_1,\ldots,w_n)$ is strongly log-concave.}
 \]
Theorem \ref{TheoremProperties} (2) and Corollary \ref{CorollaryBasisGenerating} show that
\[
\text{the polynomial $\mathrm{B}_f(w_0,w_1,\ldots,w_1)=\sum_{k=0}^n \mathrm{b}_k(f) w_0^{n-k}w_1^k$ is Lorentzian,}
\]
which, by Theorem \ref{TheoremCharacterization},  is equivalent to the condition
\[
 \text{$\frac{\mathrm{b}_0(f)}{ {n \choose 0}},\frac{\mathrm{b}_1(f)}{ {n \choose 1}},\ldots,\frac{\mathrm{b}_n(f)}{ {n \choose n}}$ is a log-concave sequence with no internal zeros.}
 \]
This proves Theorems \ref{Continuous} and \ref{Discrete}.
\end{proof}

\section{Problems}\label{SectionProblems}


\subsection{} 

We may define the homogeneous multivariate Tutte polynomial of $\operatorname{r}:2^{[n]} \to \mathbb{R}$ by
\[
\mathrm{Z}_{p,\operatorname{r}}(w_0,w_1,\ldots,w_n) \vcentcolon=\sum_{S \subseteq [n]} p^{-\operatorname{r}(S)} w_0^{n-|S|}\prod_{i \in S} w_i,
\]
where $p$ is a real parameter.\footnote{Compare the notion of universal Tutte character for submodular functions \cite[Section 8.2]{DFM18} and its multivariate version \cite[Section 4.3]{DFM18}.}
We consider the set of functions on $2^{[n]}$ with the Lorentzian property
\[
\mathscr{L}_n \vcentcolon=\Big\{\operatorname{r}:2^{[n]} \to \mathbb{R}  \mid \text{$\mathrm{Z}_{p,\operatorname{r}}$ is Lorentzian for any positive $p \le 1$}\Big\}.
\]
Which functions $\operatorname{r}:2^{[n]} \to \mathbb{R}$ belong to the set $\mathscr{L}_n$?

\begin{proposition}\label{PropositionFlagMatroid}
Let $\mathscr{M}=(\mathrm{M}_1,\ldots,\mathrm{M}_\ell)$ be a flag matroid on $[n]$.
For any
 real number $c_0$ and nonnegative real numbers $c_1,\ldots,c_\ell$, 
we have
\[
c_0+c_1  \operatorname{rk}_{\mathrm{M}_1}+\cdots+c_\ell  \operatorname{rk}_{\mathrm{M}_\ell} \in \mathscr{L}_n. 
\]
\end{proposition}

\begin{proof}
Theorem \ref{MainTheorem} proves the statement when $c_0$ is zero.
For the general case, note that the
 homogeneous multivariate Tutte polynomial of $c_0+\operatorname{r}$
 is a positive multiple of $\mathrm{Z}_{p,\operatorname{r}}$ for any positive $p$ and $\operatorname{r}:2^{[n]} \to \mathbb{R}$.
\end{proof}

\begin{remark}
A \emph{polymatroid} on $[n]$ is a function $\operatorname{rk}:2^{[n]} \to \mathbb{R}$ satisfying the following conditions \cite[Chapter 18]{Wel76}:
\begin{enumerate}[--]\itemsep 5pt
\item When $S=\varnothing$, we have $\operatorname{rk}(S) =0$.
\item When $S_1 \subseteq S_2  \subseteq [n]$,  we have
$\operatorname{rk}(S_1) \le  \operatorname{rk}(S_2)$.
\item When $S_1\subseteq [n]$, $S_2 \subseteq [n]$, we have
$\operatorname{rk}(S_1 \cup S_2)+ \operatorname{rk}(S_1\cap S_2) \le  \operatorname{rk}(S_1)+ \operatorname{rk}(S_2)$.
\end{enumerate}
For example, nonnegative linear combinations of the rank functions of the constituents of a flag matroid are polymatroids.
However,  a polymatroid on $[n]$ need not be in  $\mathscr{L}_n$.
\end{remark}

\begin{remark}
Let $\mathrm{M}$ and $\mathrm{N}$ be matroids on $[n]$.
The identity function of $[n]$ is said to be a \emph{weak map} from $\mathrm{M}$ to $\mathrm{N}$ if
any one of the following equivalent conditions hold  \cite{KN86}:
\begin{enumerate}[--]\itemsep 5pt
\item For any $S \subseteq [n]$, we have $\operatorname{rk}_{\mathrm{N}}(S) \le \operatorname{rk}_{\mathrm{M}}(S)$.
\item Every independent set of $\mathrm{N}$ is an independent set of $\mathrm{M}$.
\item Every circuit of $\mathrm{M}$ contains a circuit of $\mathrm{N}$.
\end{enumerate}
For example, the identity function of $[n]$ is a weak map from $\mathrm{M}$ to $\mathrm{N}$ when $(\mathrm{N},\mathrm{M})$ is a flag matroid. 
However, nonnegative linear combinations of the rank functions of $\mathrm{M}$ and $\mathrm{N}$ need not be in $\mathscr{L}_n$ when  the identity function of $[n]$ is a weak map  from $\mathrm{M}$ to $\mathrm{N}$.
\end{remark}

\begin{example}\label{Counterexample}
Let  $\mathrm{M}$ and $\mathrm{N}$ be matroids on $[3]$ with the sets of bases
\[
\mathscr{B}(\mathrm{M})=\{\{1,2\},\{1,3\}\}
\ \ \text{and} \ \
\mathscr{B}(\mathrm{N})=\{\{1\},\{2\}\}.
\]
The function
 $\operatorname{r}\vcentcolon=\operatorname{rk}_\mathrm{M}+\operatorname{rk}_\mathrm{N}$
is a polymatroid and the identity function of $[3]$ is a weak map from $\mathrm{M}$ to $\mathrm{N}$.
The homogeneous multivariate Tutte polynomial of $\operatorname{r}$
satisfies
\[
\lim_{p \to 0} \mathrm{Z}_{p,\operatorname{r}}(1,p^2w_1,p^2w_2,pw_3)=1+w_1+w_2+w_3+w_1w_3.
\]
The  right-hand side is not log-concave around $(w_1,w_2,w_3)=(1,1,1)$,
and hence $\operatorname{r}$ is not in $\mathscr{L}_3$.
\end{example}

The notion of $\mathrm{M}^\natural$-concavity, which equivalent to the gross substitutes property in mathematical economics \cite{RvGP02}, plays a central role in discrete convex analysis  \cite[Chapter 6]{Mur03}.
According to the  characterization  in \cite{FY03},
a function $\operatorname{r}:2^{[n]} \to \mathbb{R}$ is \emph{$\mathrm{M}^\natural$-concave} if and only if the following conditions are satisfied:
\begin{enumerate}[(1)]\itemsep 5pt
\item For any $S \subseteq [n]$ and any distinct $i,j \in [n]$, we have
\[
\operatorname{r}(S \cup i \cup j)+\operatorname{r}(S) \le \operatorname{r}(S \cup i)+\operatorname{r}(S \cup j).
\]
\item For any $S \subseteq [n]$ and any distinct $i,j,k \in [n]$,
the maximum among the three values 
\[
\operatorname{r}(S \cup j \cup k)+\operatorname{r}(S \cup i), \ \ 
\operatorname{r}(S \cup i \cup k)+\operatorname{r}(S \cup j), \ \ 
\operatorname{r}(S \cup i \cup j)+\operatorname{r}(S \cup k)
\]
is attained by at least two of them.
\end{enumerate}
For example,  the function $\operatorname{r}$ in Example \ref{Counterexample} is $\mathrm{M}^\natural$-concave.

\begin{proposition}\label{PropositionConcave}
Any function in  $\mathscr{L}_n$ is  $\mathrm{M}^\natural$-concave.
In particular, $\mathscr{L}_n$ is contained in the cone of submodular functions on $2^{[n]}$.
\end{proposition}

 Propositions \ref{PropositionFlagMatroid} and \ref{PropositionConcave} together imply that
nonnegative linear combinations of the rank functions of the constituents of a flag matroid are $\mathrm{M}^\natural$-concave.
This recovers a theorem of Shioura \cite[Theorem 3]{Shi12}.

\begin{proof}
A function $\operatorname{r}:2^{[n]} \to \mathbb{R}$ is $\mathrm{M}^\natural$-concave if and only if its homogenization is an $\mathrm{M}$-concave function on $\mathbb{N}^n$ \cite[Chapter 6]{Mur03}.
Therefore, by \cite[Section 8]{BH19},  the function $\operatorname{r}$ is $\mathrm{M}^\natural$-concave if and only if
\[
\text{$\sum_{S \subseteq E} \frac{1}{|n-|S||!} \ p^{-\operatorname{r}(S)}   w_0^{n-|S|} \prod_{i \in S} w_i$ is Lorentzian for any positive $p \le 1$.}
\]
By \cite[Section 6]{BH19}, we know that the linear operator 
\[
w_0^k \hspace{0.5mm}\prod_{i \in S} w_i \longmapsto \frac{1}{k!} \hspace{0.5mm}\ w_0^k\hspace{0.5mm} \prod_{i \in S} w_i
\]
preserves the Lorentzian property. 
Therefore, the $\mathrm{M}^\natural$-concavity of $\operatorname{r}$ follows from the condition
\[
\text{$\sum_{S \subseteq E}  p^{-\operatorname{r}(S)}   w_0^{n-|S|} \prod_{i \in S} w_i$ is Lorentzian for any positive $p \le 1$.} \qedhere
\]
\end{proof}


\subsection{}

Let $\mathbb{F}$ be an algebraically close field, and let
$h \in \mathbb{R}[w_0,w_1,\ldots,w_n]$ be a homogeneous polynomial of degree $d$.
We say that $h$ is a \emph{volume polynomial over $\mathbb{F}$}
if there are
nef divisors $\mathrm{H}_0,\mathrm{H}_1,\ldots,\mathrm{H}_n$ on a $d$-dimensional irreducible projective variety $Y$ over $\mathbb{F}$
that satisfy
\[
h=(w_0\mathrm{H}_0+w_1\mathrm{H}_1+\cdots+w_n\mathrm{H}_n)^d,
\]
where the intersection product of $Y$ is used to expand the right-hand side.\footnote{For nef divisors and intersection products, see \cite[Chapter 1]{Laz04}.}
Volume polynomials over $\mathbb{F}$ are prototypical examples of Lorentzian polynomials \cite[Section 10]{BH19}.

Let $\mathscr{R}_\mathbb{F}:\mathsf{Mat}(\mathbb{F}) \to \mathsf{Mat}$ be the functor
in Remark \ref{Remark:Linear}, and let $f$ be any morphism in $\mathsf{Mat}(\mathbb{F})$.
Corollary \ref{CorollaryBasisGenerating} shows that the homogeneous basis generating polynomial of 
$\mathscr{R}_\mathbb{F}(f)$ is a Lorentzian polynomial.

\begin{conjecture}\label{ConjectureRealization}
The homogeneous basis generating polynomial of $\mathscr{R}_\mathbb{F}(f)$ is a volume polynomial over $\mathbb{F}$.
\end{conjecture}

Let $\varphi:E\to W$ be any object in $\mathsf{Mat}(\mathbb{F})$.
In \cite[Section 4]{HW17}, the authors construct a collection of nef divisors $(\mathrm{H}_i)_{i \in E}$
on an irreducible projective variety $Y$ such that
\[
\sum_{S \in \mathscr{B}(\mathrm{M}(\varphi))} \prod_{i \in S} w_i =\Big(\sum_{i \in E} w_i\mathrm{H}_i\Big)^{\dim Y}.
\]
The construction can be used to verify Conjecture \ref{ConjectureRealization} when $f$ is the identity morphism of $\varphi$.

\end{document}